\documentclass[13pt,a4paper,oneside]{amsart}
\usepackage{cases}
\usepackage{cite}
\usepackage{latexsym}
\usepackage[arrow,matrix]{xy}
\usepackage{stmaryrd}
\usepackage{amsfonts}
\usepackage{amsmath,amssymb,amscd,bbm,amsthm,mathrsfs}
\usepackage{fancyhdr}
\usepackage{amsxtra,ifthen}
\usepackage{verbatim}
\usepackage[arrow,matrix]{xy}

\fontsize{9pt}{9pt}\selectfont

\newcommand{\Z}{\mathbb Z}

\newcommand{\lam}{\lambda}

     \DeclareMathOperator{\codeg}{codeg}

\theoremstyle{plain}
\newtheorem{prop}{Proposition}[section]
\newtheorem{thm}[prop]{Theorem}
\newtheorem{cor}[prop]{Corollary}
\newtheorem{lem}[prop]{Lemma}

\newtheorem{dfn}[prop]{Definition}

\theoremstyle{definition}
\newtheorem{example}[prop]{Example}
\newtheorem{remark}[prop]{Remark}

\begin{document}
\vspace*{0mm}
\setlength{\itemsep}{-2\jot}
%\fontsize{13}{\baselineskip}\selectfont
\setlength{\parskip}{0.4\baselineskip}
\title[Graded symmetric cellular algebras]{Notes on graded symmetric cellular algebras}
\author{Yanbo Li and Deke Zhao}

\address{Li: School of Mathematics and Statistics, Northeastern University at Qinhuangdao, Qinhuangdao, 066004, China}

\email{liyanbo707@163.com}

\address{Zhao: School of Applied Mathematics, Beijing Normal
University at Zhuhai, Zhuhai, 519087, China}

\email{deke@amss.ac.cn}



\begin{abstract}Let $A=\oplus_{i\in \mathbb{Z}}A_i$ be a finite dimensional graded symmetric cellular algebra with a homogeneous symmetrizing trace of degree $d$. We prove that  $A_{-d}$ contains the Higman ideal $H(A)$ of the center of $A$ and $\dim H(A)\leq \dim A_{0}$ if $d\neq 0$, and provide a semisimplicity criterion of $A$ in terms of  the centralizer of $A_0$, which is a graded version of \cite[Theorem 3.2]{L}.
\end{abstract}

\thanks{The first named author is supported by the Natural Science Foundation of Hebei
Province, China (A2017501003) and the Science and Technology support
program of Northeastern University at Qinhuangdao (No. XNK201601).
The second named author is supported partly by NSFC 11571341, 11671234.}
\maketitle

\section{Introduction}
Cellular algebras were introduced by Graham and Lehrer \cite{GL} in 1996 motivated by the work of Kazhdan and Lusztig \cite{KL}. It provides a
systematic framework for studying the representation theory of many interesting and important algebras coming from mathematics and physics, such as Schur algebras, Temperley-Lieb algebras, Brauer algebras \cite{GL}, partition algebras \cite{Xi1}, Birman-Wenzl algebras \cite{Xi2},  Hecke algebras of finite types \cite{G}, and so on.

The ($\mathbb{Z}$-)gradings is a subtle structure on a finite dimensional algebra, which plays an important role in Lie theory and the representation theory (ref. \cite{BGS, CPS, F-Zimmer, R}).
Motivated by the works of Brundan, Kleshchev (and Wang) \cite{BK, BK2, BKW}, Hu and Mathas \cite{HM} introduced graded cellular algebras, which include  the Khovanov
diagram algebras and their quasi-hereditary covers \cite{BS,BS2}, the
level two degenerate cyclotomic Hecke algebras \cite{AST}, graded walled Brauer algebras \cite{BS3}, Temperley-Lieb algebras of type $A$ and $B$ \cite{PR}, etc (see references in \cite{HM2, HM3}).

The Auslander-Reiten conjecture \cite{ARS} claims that if the stable categories of two Artin algebras are equivalent then they have the same number of non-projective simple modules up to isomorphism. Recently, Liu, Zhou and Zimmermann's work \cite{LZZ} indicates that the projective center is the main obstruction to attack the Auslander-Reiten conjecture, which is exactly the Higman ideal for a symmetric algebra.  It is natural and interesting to investigate the Higman ideal of the center of a symmetric algebra.

The aim of this note is to study the Higman ideal of a graded symmetric cellular algebra by applying the dual basis method which has been used in \cite{L, L1, L2, L3, LX, LZ}. More precisely,
assume that $A=\oplus_{i\in \mathbb{Z}}A_i$ is a finite dimensional graded symmetric cellular algebra over a field $K$ with a homogeneous symmetrizing trace $\tau$ of degree $d$. Our first main result claims that the Higman ideal of the center $\mathcal{Z}(A)$ of $A$ is
contained in $A_{-d}$, and
$\dim H(A)\leq \dim A_{0}$ whenever $d\neq 0$  (see Theorem~\ref{3.4}). Secondly, we provide a  semisimplicity criterion for finite dimensional graded symmetric cellular algebras in terms  of the centralizer of $A_0$ (see Theorem~\ref{Them:semisimple}), which is a graded version of \cite[Theorem 3.2]{L}.

The note is organized as follows. In Section 2, we  review quickly some known results on symmetric algebras, graded algebras and cellular algebras. In Section~3, we study the Higman ideal of finite dimensional graded symmetric cellular algebras and prove the first main result.  Section~4 devotes  to investigate the centralizer of $A_0$ and gives a graded version of \cite[Thoerem 3.2]{L}.  The semisimplicity criterion for finite dimensional graded symmetric cellular algebras is provided in the last section.

\section{Preliminaries}

In this section, we recall some facts on symmetric algebras,
graded algebras and cellular algebras, and introduce the graded symmetric algebras.

\subsection{Symmetric algebras}
Let $K$ be a field and let $A$ be a finite dimensional $K$-algebra.
Recall that a bilinear form $f: A\times A\rightarrow K$ is {\it
non-degenerate} if the determinant of the matrix $(f(x_{i},\,x_{j}))_{x_{i},\,\, x_{j}\in \mathscr{B}}$ is
invertible for some basis $\mathscr{B}=\{x_1, \cdots, x_n\}$ of $A$ and  $f$ is {\it
associative} if $f(ab,c)=f(a,bc)$ for all $a, b, c\in A$.  We say
that $A$ is a {\em symmetric algebra} if there is a
non-degenerate associative bilinear form $f$ on $A$ such that
$f(x,y)=f(y,x)$ for all $x, y\in A$. Note that if $A$ is a symmetric algebra, then we can
define a linear map $\tau$ by
  \begin{equation*}\tau: A\rightarrow K\quad\quad a\mapsto
f(a,1),\end{equation*} which is  called the {\it symmetrizing trace} of $A$ induced by $f$.

Now let $A$ be a finite dimensional symmetric algebra with a basis $\mathscr{B}=\{x_1, \cdots, x_n\}$ and denote by $\mathscr{D}=\{y_1, \cdots, y_n\}$ the {\em dual basis} of $\mathscr{B}$, that is, $\mathscr{D}$ is a basis of $A$ satisfying $\tau(x_iy_j)=\delta_{ij}$ for all $i, j=1,\ldots,n$. Then the {\em Higman ideal} $H(A)$ of the center $\mathcal{Z}(A)$ of $A$ is
$$H(A):=\biggl\{\sum_ix_iay_i\biggl| \,a\in A\biggr.\biggr\},$$
which is independent of the choice of $\tau$ and the basis of $A$.

For arbitrary $1\leq i,j \leq n$, write
$x_{i}x_{j}=\sum_{k}r_{ijk}x_{k},$ where $r_{ijk}\in K$. The first named author proved the following lemma.

\begin{lem}[\protect{\cite[Lemma 2.2]{L2}}]\label{2.0} Let $A$ be a symmetric $K$-algebra. Then the following hold:
$$x_{i}y_{j}=\sum_{k}r_{kij}y_{k};\,\,\,\,\,y_{i}x_{j}=\sum_{k}r_{jki}y_{k}.$$
\end{lem}

\subsection{Graded symmetric algebras}

By a {\em graded space} we mean a
$\mathbb{Z}$-graded $K$-space $V$, namely a $K$-space with a
decomposition into subspaces $V=\bigoplus_{i\in \mathbb{Z}}V_i$. A
nonzero element $v$ of $V_i$ is said to be a {\em homogeneous element} of degree $i$ and
denoted by ${\rm deg}(v)=i$. We will view the field  $K$ as a
graded space concentrated in degree $0$.
Given two graded spaces $V$ and $W$, the $K$-space
$\mathrm{Hom}_K(V,W)$ of all $K$-linear maps from $V$ to $W$ is
a graded space with $\mathrm{Hom}_K(V,W)_i$ consisting
of all the $K$-linear maps $\alpha: V\rightarrow W$ such that
$\alpha(V_j)\subseteq W_{j+i}$ for all $i, j\in \mathbb{Z}$.
Nonzero element of $\mathrm{Hom}_K(V,W)_i$ will be called a
\textit{homogeneous map} of degree $i$.

By a \textit{graded algebra} $A$ we always mean a finite dimensional
$\mathbb{Z}$-graded associative $K$-algebra with identity, that
is, $A$ is a  graded space $A=\bigoplus_{i\in \mathbb{Z}}A_i$ such
that $A_iA_j\subseteq A_{i+j}$ for all $i, j\in \mathbb{Z}$.

\begin{dfn}\label{2.1}
A graded algebra $A$ is said to be a \textbf{graded  symmetric algebra} if
there is a homogeneous symmetrizing trace $\tau: A\rightarrow K$ of
degree $d$ for some $d\in \mathbb{Z}$.
\end{dfn}

\subsection{Cellular algebras}
Now we recall the definitions of cellular algebras, Gram matrices and cell modules.
\begin{dfn}[\protect{\cite[Definition~1.1]{GL}}]\label{2.2}
Let $R$ be a commutative ring with identity. An associative unital
$R$-algebra is called a \textbf{cellular algebra} with cell datum $(\Lambda, M, C, \ast)$ if the
following conditions are satisfied:

\begin{enumerate}
\item[(GC1)] The finite set $\Lambda$ is a poset. Associated with
each $\lam\in\Lambda$, there is a finite set $M(\lam)$. The
algebra $A$ has an $R$-basis $\{C_{S,T}^\lam \mid S,T\in
M(\lam),\lam\in\Lambda\}$.

\item[(GC2)] The map $\ast$ is an $R$-linear anti-automorphism of $A$
such that $(C_{S,T}^\lam)^{\ast}= C_{T,S}^\lam $ for
all $\lam\in\Lambda$ and $S,T\in M(\lam)$.

\item[(GC3)] If $\lam\in\Lambda$ and $S,T\in M(\lam)$, then for any
element
$a\in A$, we have
$$aC_{S,T}^\lam\equiv\sum_{S^{'}\in
M(\lam)}r_{a}(S',S)C_{S^{'},T}^{\lam} \,\,\,\,(\rm {mod}\,\,\,
A(<\lam)),$$ where $r_{a}(S^{'},S)\in R$ is independent of $T$ and
$A(<\lam)$ is the $R$-submodule of $A$ generated by
$\{C_{U,V}^\mu \mid U,\,\,V\in M(\mu),\mu<\lam\}$.
\end{enumerate}
\end{dfn}

Let $\lam\in\Lambda$. For arbitrary elements $S,T,U,V\in M(\lam)$,
Definition~\ref{2.2} implies that
$$C_{S,T}^\lam C_{U,V}^\lam \equiv\Phi(T,U)C_{S,V}^\lam\,\,\,\, (\rm mod\,\,\, A(<\lam)),$$
where $\Phi(T,U)\in R$ depends only on $T$ and $U$. It is easy to check that $\Phi(T,U)=\Phi(U,T)$ for arbitrary $T,U\in M(\lam)$.

For each $\lam\in \Lambda$, fix an order
on $M(\lam)$. The associated  {\em Gram matrix} $G(\lam)$ is the following symmetric matrix
$$G(\lam)=(\Phi(S, T))_{S,T\in M(\lam)}.$$
 Note that the determinant of $G(\lam)$ is independent of the choice of the order on
$M(\lam)$.

Given a cellular algebra $A$, we note that $A$ has a family of modules defined by its cellular structure.

\begin{dfn}[\protect{\cite[Definition 2.1]{GL}}]\label{2.3}
Let $A$ be a cellular algebra with cell datum $(\Lambda, M, C,
\ast)$. For each $\lam\in\Lambda$, the \textit{cell module} $W(\lam)$ is a free $R$-module with basis
$\{C_{S}\mid S\in M(\lam)\}$ and the left $A$-action
defined by
$$aC_{S}=\sum_{S^{'}\in M(\lam)}r_{a}(S^{'},S)C_{S^{'}}
\,\,\,\,(a\in A,\,\,S\in M(\lam)),$$ where $r_{a}(S^{'},S)$ is the
element of $R$ defined in Definition~\ref{2.2}(GC3).
\end{dfn}

\subsection{Symmetric cellular algebras}

Let $A$ be a symmetric
cellular algebra with cell datum $(\Lambda, M, C, \ast)$. Fix a symmetrizing trace $\tau$ and denote the dual basis by  $$D=\{D_{S,T}^\lam
\mid S,T\in M(\lam),\lam\in\Lambda\},$$ which is determined by
$$\tau(C_{S,T}^{\lam}D_{U,V}^{\mu})=\delta_{\lam\mu}\delta_{SV}\delta_{TU}.$$

Set $e_{\lam}=\sum_{S\in M(\lam)}C_{S,S}^{\lam}D_{S,S}^{\lam}$. The first named author \cite{L} introduced the following ideal $L(A)$ of $\mathcal{Z}(A)$
\begin{eqnarray*}
L(A)&:=& \biggl\{\sum_{\lam\in\Lambda} r_{\lam}e_{\lam}\biggl|\, r_{\lam}\in K\biggr.\biggr\},
\end{eqnarray*}
and proved that $H(A)\subseteq L(A)$.

For any $\lam, \mu\in \Lambda$, $S,T\in M(\lam)$, $U,V\in M(\mu)$,
write
$$C_{S,T}^{\lam}C_{U,V}^{\mu}=\sum\limits_{\epsilon\in\Lambda,X,Y\in M(\epsilon)}
r_{(S,T,\lam),(U,V,\mu),(X,Y,\epsilon)}C_{X,Y}^{\epsilon}.$$
The following lemma is important to our purpose.

\begin{lem}[\protect{\cite[Lemma 3.1]{L2}}]\label{2.4}
Let $A$ be a symmetric cellular algebra with a cell datum $(\Lambda,
M, C, \ast)$ and $\tau$ a given symmetrizing trace. For arbitrary $\lam,\mu\in\Lambda$ and $S,T,P,Q\in M(\lam)$, $U,V\in M(\mu)$, the following hold:
\begin{enumerate}
\item[(1)]$D_{U,V}^{\mu}C_{S,T}^{\lam}=\sum\limits_{\epsilon\in
\Lambda, X,Y\in
M(\epsilon)}r_{(S,T,\lam),(Y,X,\epsilon),(V,U,\mu)}D_{X,Y}^{\epsilon}.$
\item[(2)]$C_{S,T}^{\lam}D_{U,V}^{\mu}=\sum\limits_{\epsilon\in
\Lambda, X,Y\in
M(\epsilon)}r_{(Y,X,\epsilon),(S,T,\lam),(V,U,\mu)}D_{X,Y}^{\epsilon}.$
\item[(3)] $C_{S,T}^{\lam}D_{T,Q}^{\lam}=C_{S,P}^{\lam}D_{P,Q}^{\lam}.$
\item[(4)]$D_{T,S}^{\lam}C_{S,Q}^{\lam}=D_{T,P}^{\lam}C_{P,Q}^{\lam}.$
\item[(5)]$C_{S,T}^{\lam}D_{P,Q}^{\lam}=0 \text{ if }  T\neq P.$
\item[(6)]$D_{P,Q}^{\lam}C_{S,T}^{\lam}=0 \text{ if } Q\neq S.$
\item[(7)]$C_{S,T}^{\lam}D_{U,V}^{\mu}=0 \text{ if } \mu\nleq \lam.$
\item[(8)]$D_{U,V}^{\mu}C_{S,T}^{\lam}=0 \text{ if }  \mu\nleq\lam.$
\end{enumerate}
\end{lem}

Denote by $G'(\lam)$ the Gram matrices
defined by the dual basis. The first named author \cite[Lemma~3.6]{L2} showed $G(\lam)G'(\lam)=k_{\lam}E$ for some $k_{\lam}\in K$, where $E$ is the identity matrix.

The following facts on the constants $k_{\lam}$, $\lam\in \Lambda$ will be used later.

\begin{lem}[\protect{\cite[Lemma 3.3]{L2}}]\label{2.6}
Let $A$ be a symmetric cellular algebra. Then $(C_{S,S}^{\lam}D_{S,S}^{\lam})^2=k_{\lam}C_{S,S}^{\lam}D_{S,S}^{\lam}$ for arbitrary $\lam\in\Lambda$ and $S\in M(\lam)$.
\end{lem}

\begin{lem}[\protect{\cite[Theorem 4.4]{LX}}]\label{2.5}
Let $A$ be a symmetric cellular algebra with a
cell datum $(\Lambda, M, C, \ast)$. Then for any $\lam\in\Lambda$,
the cell module $W(\lam)$ is projective if and only if $k_\lam\neq
0$.
\end{lem}

\begin{lem}[\protect{\cite[Corollary 4.7]{L2}}]\label{2.8}
Let $A$ be a finite dimensional symmetric cellular algebra. Then the following statements are equivalent.
\begin{enumerate}
\item[(1)]$A$ is semisimple.
\item[(2)]$k_{\lam}\neq 0$ for all $\lam\in\Lambda$.
\item[(3)]$\{C_{S,T}^{\lam}D_{T,T}^{\lam}\mid\lam\in\Lambda, S,T\in M(\lam)\}$ is a basis of $A$.
\end{enumerate}
\end{lem}

\section{Homogeneous symmetrizing trace and Higman ideal}

Let $A$ be a cellular algebra. Following Hu and Mathas \cite{HM},  we say that $A$ is a {\em  graded cellular algebra} if $A$ is in
addition a $\mathbb{Z}$-graded algebra satisfying the following
condition:

\begin{enumerate}
\item[(GC$_d$)] Let ${\rm deg}:
\coprod_{\lam\in\Lambda}M(\lam)\rightarrow \mathbb{Z}$ be a
function. Each basis element $C_{S,T}^\lam$ is homogeneous of
degree ${\rm deg}(C_{S,T}^\lam )= {\rm deg}(S)+{\rm deg}(T)$.
\end{enumerate}

A graded cellular algebra $A$ is called a {\em graded symmetric cellular algebra} if $A$
is in addition a graded symmetric algebra.  Note that a finite
dimensional semi-simple algebra is a graded symmetric cellular algebra with a non-trivial
grading (see Section 5 for details), which is a generalization of \cite[Example 2.2]{HM}.

The following is an example of non-semisimple graded symmetric cellular algebras.

\begin{example}\label{3.1}
Let $K$ be a field and $A=K\oplus Kx$ with ${\rm deg}(x)=2$ and $x^2=0$.
Then $A$ is a graded symmetric algebra with a non-degenerate trace
form $\tau(1)=0$ and $\tau(x)=1$. It is a homogeneous trace form
of degree $-2$. It is easy to check that the dual basis of $\{1,
x\}$ is $\{x, 1\}$. Thus the dual basis is homogeneous too.
\end{example}

The following easy verified fact is useful, which implies that the
dual basis of the homogeneous basis of graded symmetric
 algebras with respect to a homogeneous symmetrizing trace must be
homogeneous.

\begin{lem}\label{3.2}
Let $A$ be a graded symmetric algebra with homogeneous symmetrizing trace
$\tau$ of degree $d$ and let $\mathscr{B}=\{x_i\mid i\in 1,\cdots, n\}$ be a
homogeneous basis of $A$. Then the dual basis
$\mathscr{D}=\{y_i\mid  \tau(x_iy_j)=\delta_{ij},\,\, i, j\in 1,\cdots, n\}$ of $A$  is a homogeneous basis with
$${\rm deg}(x_i)+{\rm deg}(y_i)=-d.$$
\end{lem}

Combining Lemma \ref{3.2} with \cite[Theorem 2.5]{L1}, we get the following

\begin{cor}\label{3.3}
Let $A$ be a graded symmetric cellular algebra with a homogeneous symmetrizing trace $\tau$ of degree $d$.
Then the dual basis is a graded cellular basis if and only if $d$ is even and $\tau(a)=\tau(a^{\ast})$ for all $a\in A$.
\end{cor}
\begin{proof}
$``\Rightarrow"$ It follows from \cite[Theorem 2.5]{L1} that if the dual basis $\{D_{S,T}^{\lam}\mid\lam\in\Lambda, S,T\in M(\lam)\}$ is cellular, then $\tau(a)=\tau(a^{\ast})$ for all $a\in A$. Since the dual basis is graded cellular, the condition ($GC_d$) implies that there is a function ${\codeg}:
\coprod_{\lam\in\Lambda}M(\lam)\rightarrow \mathbb{Z}$ such that $\deg (D_{S,S}^{\lam})=2\codeg (S)$. Thanks to Lemma~\ref{3.2},  $d=-2(\deg(S)+\codeg(S))$ is even.

$``\Leftarrow"$ Suppose $\tau(a)=\tau(a^{\ast})$ for all $a\in A$.  Then \cite{G2} or \cite[Theorem 2.5]{L1}, and Lemma \ref{3.2} imply that the dual basis is cellular and homogeneous. Now we define $${\codeg}: \coprod_{\lam\in\Lambda}M(\lam)\rightarrow \mathbb{Z},\qquad\quad\codeg (S)=-\deg(S)-\frac{d}{2}.$$
Applying Lemma \ref{3.2}, $\deg(D_{S,T}^{\lam})=-d-\deg(S)-\deg(T)=\codeg (S)+\codeg (T)$. This completes the proof.
\end{proof}

Now we are in a place to give the main result of this section.

\begin{thm}\label{3.4}
Let $A$ be a finite dimensional graded symmetric cellular algebra with a homogeneous symmetrizing
trace $\tau$ of degree $d\neq 0$. Then
\begin{enumerate}
\item[(1)] If $\rho$ is a homogeneous symmetrizing
trace of degree $d'$, then $d=d'$.
\item[(2)] All cell modules are non-projective.
\item[(3)] $H(A)\subseteq L(A)\subseteq A_{-d}$ and $\dim H(A)\leq \dim A_{0}$.
\end{enumerate}
\end{thm}

\begin{proof}
(1) Let $\{D_{S,T}^{\lam}\mid \lam\in\Lambda,\,\, S, T\in M(\lam)\}$
and $\{d_{S,T}^{\lam}\mid \lam\in\Lambda,\,\, S, T\in M(\lam)\}$ be
dual bases determined by $\tau$ and $\rho$, respectively. It follows
from Lemma \ref{3.2} that $\deg(C_{S, S}^{\lam}D_{S, S}^{\lam})=-d$
and $\deg(C_{S, S}^{\lam}d_{S, S}^{\lam})=-d'$ for arbitrary $S\in M(\lam)$. According to
\cite[Lemma 2.3]{L2}, $$d_{S, S}^{\lam}=\sum\limits_{\varepsilon\in\Lambda, X, Y\in M(\varepsilon)}\tau(C_{X,Y}^{\varepsilon}d_{S, S}^{\lam})D_{Y,X}^{\varepsilon}.$$ Therefore
\begin{eqnarray*}
  C_{S, S}^{\lam}d_{S, S}^{\lam} &=& \sum\limits_{\varepsilon\in\Lambda, X, Y\in M(\varepsilon)}\tau(C_{X,Y}^{\varepsilon}d_{S, S}^{\lam})C_{S, S}^{\lam}D_{Y,X}^{\varepsilon} \\
   &=& \sum\limits_{ X\in M(\lam)}\tau(C_{X,S}^{\lam}d_{S, S}^{\lam})C_{S, S}^{\lam}D_{S,X}^{\lam}\\
   &=& \sum\limits_{ X\in M(\lam)}\tau(d_{S, S}^{\lam}C_{X,S}^{\lam})C_{S, S}^{\lam}D_{S,X}^{\lam}\\
   &=&\tau(C_{S,S}^{\lam}d_{S, S}^{\lam})C_{S, S}^{\lam}D_{S,S}^{\lam},
\end{eqnarray*}
where the second and the last equality follow by applying Lemma~\ref{2.4}. Clearly, $\tau(C_{S,S}^{\lam}d_{S, S}^{\lam})\neq 0$ and this forces $d=d'$.

(2) Suppose that $W(\lam)$ is a projective cell module. Then Lemma~\ref{2.5}
implies $k_{\lam}\neq 0$. Thus $k_{\lam}^{-1}C_{S,S}^\lam D_{S,S}^\lam$ is an idempotent of $A$ for arbitrary $S\in M(\lam)$ due to Lemma \ref{2.6}.
This forces ${\deg}(C_{S,S}^\lam)+{\deg}(D_{S,S}^\lam)=0$,
while Lemma \ref{3.2} shows
${\deg}(C_{S,S}^\lam)+{\deg}(D_{S,S}^\lam)=-d\neq 0$.  It is a
contradiction and we complete the proof.

(3) According to Lemma~\ref{3.2}, we have $L(A)\subseteq A_{-d}$. Note that
Li \cite{L} proved that $H(A)\subseteq L(A)$. So $H(A)\subseteq A_{-d}$ and we only  need to prove $\dim H(A)\leq \dim A_{0}$. In fact, For each $C_{X,Y}^\varepsilon$ with $\deg(C_{X,Y}^\varepsilon)\neq 0$, it follows from $H(A)\subseteq A_{-d}$ and Lemma \ref{3.2} that $$\sum_{\lam\in\Lambda, \,\,S,T\in M(\lam)} C_{S, T}^\lam C_{X,Y}^\varepsilon D_{T, S}^\lam=0.$$ This implies that $H(A)$ is a $K$-span of $$\biggl\{\sum_{\lam\in\Lambda, \,\,S,T\in M(\lam)} C_{S, T}^\lam C_{X,Y}^\varepsilon D_{T, S}^\lam\mid C_{X,Y}^\varepsilon\in A_0\biggr\}$$ and consequently, $\dim H(A)\leq \dim A_{0}$.
\end{proof}

\begin{remark}\label{3.5}
Hu and Mathas \cite{HM} proved that the blocks
$\mathscr{H}_{\beta}^{\Lambda}$ of the cyclotomic Hecke algebras of
type $G(m, 1, n)$ are graded symmetric cellular algebras with
homogeneous trace form of degree $-2{\rm def}\beta$.  Now
Theorem~\ref{3.4} implies that the degree $-2{\rm def}\beta$ is the only one which
makes $\mathscr{H}_{\beta}^{\Lambda}$ to be graded symmetric cellular and the (ungraded) cell modules of
$\mathscr{H}_{\beta}^{\Lambda}$ are non-projective
when ${\rm def}\beta\neq 0$.
\end{remark}

Let us remark that the equality in Theorem \ref{3.4}(3) may be held. The following is an example given in \cite{L},  which is a graded symmetric cellular algebra.

\begin{example}\label{3.6}
Let $K$ be a field with ${\rm Char}K\nmid n+1$ and let $Q$ be
the following quiver
$$\xymatrix@C=13mm{
  \bullet \ar@<2.5pt>[r]^{\alpha_1}  & \bullet \ar@<2.5pt>[r]^(0.4){\alpha_2}
  \ar@<2.5pt>@[r][l]^(1){1}^{\alpha_1'}^(0){2}
  &\bullet\ar@<2.5pt>@[r][l]^(0.20){3}^(0.6){\alpha_2'}\cdots
  \bullet\ar@<2.5pt>[r]^(0.6){\alpha_{n-1}} & \bullet\ar@<2.5pt>@[r][l]^(0.35){\alpha_{n-1}'}^(0.75){n-1}^(0){n}\\
}$$ with relation $\rho$ given as follows:

\begin{enumerate}
\item[(1)] all paths of length $\geq 3$;

\item[(2)] $\alpha'_{i}\alpha_{i}-\alpha_{i+1}\alpha'_{i+1}$,
$i=1,\cdots, n-2$;

\item[(3)] $\alpha_{i}\alpha_{i+1}$, $\alpha'_{i+1}\alpha'_{i}$,
$i=1,\cdots, n-2$.
\end{enumerate}
Then $A=K(Q,\rho)$ is a graded algebra in a natural way. Now we define the homogeneous symmetrizing trace $\tau$ of $A$ by
\begin{enumerate}
\item[(1)] $\tau(e_{1})=\cdots=\tau(e_{n})=0$;

\item[(2)] $\tau(\alpha_{i})=\tau(\alpha'_{i})=0$, $i=1,\cdots, n-1$;

\item[(3)]
$\tau(\alpha_{i}\alpha'_{i})=\tau(\alpha'_{i}\alpha_{i})=1$,
$i=1,\cdots, n-1$.
\end{enumerate}
Then the degree of $\tau$ is $-2$, $\tau(a)=\tau(a^\ast)$ for all $a\in A$, and $A_0=\{\sum k_ie_i\mid k_i\in K\}$.
Furthermore, $A$ is a graded symmetric cellular algebra with a homogeneous cellular basis
$\{C_{i,j}^k\mid 1\leq k\leq n+1, 1\leq i,j\leq 2\}$ as follows:
\[ \begin{matrix}
\begin{matrix} e_{1}\end{matrix} ;&
\begin{matrix} \alpha_{1}\alpha_{1}^{'} & \alpha_{1}\\ \alpha_{1}^{'} &
e_{2}\end{matrix} ; &
\begin{matrix} \alpha_{2}\alpha_{2}^{'} & \alpha_{2}\\ \alpha_{2}^{'} &
e_{3}\end{matrix} ;& \cdots ;&
\begin{matrix} \alpha_{n-1}\alpha_{n-1}^{'} & \alpha_{n-1}\\ \alpha_{n-1}^{'} &
e_{n}\end{matrix} ;&
\begin{matrix} \alpha_{n-1}^{'}\alpha_{n-1}\end{matrix},
\end{matrix} \]
and the Higman ideal $H(A)$ is generated by
\begin{align*}
&\left\{2\alpha_{1}\alpha_{1}^{'}+\alpha_{2}\alpha_{2}^{'},
\alpha_{1}\alpha_{1}^{'}+2\alpha_{2}\alpha_{2}^{'}+\alpha_{3}\alpha_{3}^{'},
\alpha_{2}\alpha_{2}^{'}+2\alpha_{3}\alpha_{3}^{'}+\alpha_{4}\alpha_{4}^{'},
\cdots,\right.\\
&\left.\qquad\alpha_{n-3}\alpha_{n-3}^{'}+2\alpha_{n-2}\alpha_{n-2}^{'}+\alpha_{n-1}\alpha_{n-1}^{'},
\alpha_{n-2}\alpha_{n-2}^{'}+2\alpha_{n-1}\alpha_{n-1}^{'}\right\}.
\end{align*}
It is easy to check that $H(A)\subset A_2$ and $\dim H(A)=\dim A_0=n$.
Furthermore, a direct computation yields that $C_{i,i}^kD_{i,i}^k\in\mathcal {Z}_A(A_0)$ for $i=1,2$ and $k=1, \ldots, n+1$.
However, these elements can not form the whole $\mathcal {Z}_A(A_0)$.

\end{example}

\section{Centralizer of $A_0$}
This section devotes to give a graded version of \cite[Theorem 3.2]{L}. Throughout this section $A$ is a finite dimensional graded symmetric cellular $K$-algebra with a homogeneous symmetrizing trace $\tau$ of degree $d$.

For any integer $c$, it is interesting to consider the following subset $H_c(A)$ of the Higman ideal $H(A)$ of $A$ $$H_c(A):=\biggl\{\sum\limits_{{\rm deg}(x_i)=c} x_iay_i\mid a\in A\biggr\}.$$

\begin{prop}\label{4.1}Let $H_{\rm gr}(A)$ be the $K$-span of all $H_c(A)$ and $\mathcal {Z}_A(A_0)$ the centralizer of $A_0$ in $A$. Then $H_{\rm gr}(A)\subseteq\mathcal {Z}_A(A_0)$.
\end{prop}

\begin{proof}
Clearly, we only need to prove $H_c(A)\subseteq\mathcal {Z}_A(A_0)$ for each integer $c$. Assume that ${\rm deg}(x_j)=0$. Thanks to Lemma~\ref{2.0},
$$\sum\limits_{{\rm deg}(x_i)=c} x_jx_iay_i=\sum\limits_{{\rm deg}(x_i)=c}\sum\limits_kr_{jik}x_kay_i,$$
where $r_{jik}=0$ when ${\rm deg}(x_k)\neq c$. This implies that
$$\sum\limits_{{\rm deg}(x_i)=c} x_jx_iay_i=\sum\limits_{\substack{i,\,k\\{\rm deg}(x_i)={\rm deg}(x_k)=c}}r_{jik}x_kay_i. \eqno(*)$$
While Lemma\ref{2.0} implies
$$\sum\limits_{{\rm deg}(x_i)=c} x_iay_ix_j=\sum\limits_{{\rm deg}(x_i)=c}\sum\limits_k r_{jki}x_iay_k,$$
where $r_{jki}=0$ if ${\rm deg}(x_k)\neq c$. Thus
$$\sum\limits_{{\rm deg}(x_i)=c} x_iay_ix_j=\sum\limits_{\substack{i,\,k\\{\rm deg}(x_i)={\rm deg}(x_k)=c}}r_{jki}x_iay_k. \eqno(**)$$
Comparing the equalities ($*$) and ($**$), we complete the proof.
\end{proof}

Now we consider the following elements  $$e_{\lam, \,c}:=\sum\limits_{\mathrm{deg}(S)=c}C_{S,S}^{\lam}D_{S,S}^{\lam}.$$

 Applying Lemmas~\ref{2.4} and \ref{2.6}, we get the following lemma.
\begin{lem}\label{4.2'}Keep notations as above. Then $e_{\lam, \,c}e_{\mu, \,c'}=\delta_{\lam\mu}\delta_{cc'}k_{\lam}e_{\lam, \,c}.$
\end{lem}

Define $$L_{\rm gr}(A):=\biggl\{\sum\limits_{\lam\in\Lambda, \,c \,\in\,\Z} r_{\lam, \,c}e_{\lam, \,c}\mid r_{\lam, \,c}\in K\biggr\}.$$
Using the similar argument as \cite[Propositition 3.3 (1)]{L}, we can show that
$L_{\rm gr}(A)$ is independent of the choice of $\tau$.

The following fact will be used later.

\begin{lem}\label{4.2}Keep notations as above. Then
$L_{\rm gr}(A)\subseteq \mathcal {Z}_A(A_0)$.
\end{lem}

\begin{proof}
Clearly, we only need to prove $e_{\lam, \,c}\in \mathcal {Z}_A(A_0).$ Let $C_{U,V}^{\mu}$ be a basis element of degree 0. Then by Lemma \ref{2.4},
\begin{eqnarray*}
\sum_{{\rm deg}(S)=c}C_{S,S}^{\lam}D_{S,S}^{\lam}C_{U,V}^{\mu}& =&\sum_{{\rm deg}(S)=c}\sum_{\substack{\epsilon\in \Lambda,\\ X,Y\in
M(\epsilon)}}r_{(U,V,\mu),(X,Y,\epsilon),(S,S,\lam)}C_{S,S}^{\lam}D_{Y,X}^{\epsilon}\\
&=&\sum_{{\rm deg}(S)=c}\sum_{X\in
M(\lam)}r_{(U,V,\mu),(X,S,\lam),(S,S,\lam)}C_{S,S}^{\lam}D_{S,X}^{\lam}\\&=&\sum_{{\rm deg}(S)={\rm
deg}(X)=c}r_{(U,V,\mu),(X,S,\lam),(S,S,\lam)}C_{S,S}^{\lam}D_{S,X}^{\lam},
\end{eqnarray*}
where the second equality follows from Definition \ref{2.3} and Lemma~\ref{2.4}(7), the last one follows by comparing the degree of both sides.

On the other hand, we have
\begin{align*}
C_{U,V}^{\mu}\sum_{
\mathrm{deg}(S)=c}C_{S,S}^{\lam}D_{S,S}^{\lam}&=\sum_{
\mathrm{deg}(S)=c}\sum_{\substack{\epsilon\in \Lambda, \\X,Y\in
M(\epsilon)}}r_{(U,V,\mu),(S,S,\lam),(X,Y,\epsilon)}C_{X,Y}^{\epsilon}D_{S,S}^{\lam}\\
&=\sum_{\mathrm{deg}(S)=c}\sum_{X\in
M(\lam)}r_{(U,V,\mu),(S,S,\lam),(X,S,\lam)}C_{X,S}^{\lam}D_{S,S}^{\lam}\\
&=\sum_{\mathrm{deg}(S)=\mathrm{deg}(X)=c}r_{(U,V,\mu),(S,S,\lam),(X,S,\lam)}C_{X,S}^{\lam}D_{S,S}^{\lam},
\end{align*}
where the second equality and the third one follow form Lemma~\ref{2.4}(5, 7).
Thus $e_{\lam,c}\in
\mathcal {Z}_A(A_0)$ as required.
\end{proof}

\begin{cor}Keep notations as above.  If
$d\neq 0$ then $L_{\rm gr}(A)\subsetneqq \mathcal {Z}_A(A_0)$.
\end{cor}
\begin{proof}
Note that $L_{\rm gr}(A)\subseteq A_{-d}$. Hence $\mathcal {Z}(A_0)$ is not contained in $L_{\rm gr}(A)$.
\end{proof}

 The relationship between $H_{\rm gr}(A)$ and $L_{\rm gr}(A)$ is given by the following lemma, which can be proved by the similar argument of \cite[Theorem 3.2]{L}.

\begin{lem}\label{4.3}Keep notations as above. Then
$H_{\rm gr}(A)\subseteq L_{\rm gr}(A)$.
\end{lem}

Combining Lemmas~\ref{4.2} and \ref{4.3}  yields the main result of this section.

\begin{thm}\label{4.4}Keep notations as above. Then
$H_{\rm gr}(A)\subseteq L_{\rm gr}(A)\subseteq \mathcal {Z}_A(A_0)$.
\end{thm}

\section{Semisimple case}

In this section,  we give a semisimlicity criterion for a graded symmetric cellular algebra, that is, $A$ is semisimple if and only if $L_{\rm gr}(A)=\mathcal {Z}_A(A_0)$.

Now let $S_n$ be the symmetric group on $n$ letters and $A=M_n(K)$. For $\sigma_1, \sigma_2\in S_n$, we set $e_{ij}=C_{\sigma_1(i)\sigma_2(j)}$, $1\leq i, j\leq n$, where $e_{ij}$ is the $n\times n$ matrix with only one non-zero $(i,j)$-entry 1.

\begin{prop}\label{5.1}
Keep notations as the above and define $\sigma=\sigma_1\sigma_2^{-1}$. Let $``\rm deg"$ be a function from $\{1,2,\cdots, n\}$ to $\mathbb{Z}$.
Then the basis $\{C_{i,j}\mid 1\leq i,j\leq n\}$ is graded cellular if and only if
\begin{enumerate}
\item [(1)] $\sigma^2=\mathrm{id}$;
\item [(2)] ${\rm deg}(i)=-{\rm deg}(\sigma(i))$ for $1\leq i\leq n$.
\end{enumerate}
\end{prop}

\begin{proof}
$``\Longrightarrow"$ For all $1\leq j\leq n$,
 the cellularity of $C_{i,j}$ shows
\begin{eqnarray*}
e_{ij}e_{jk}&=&C_{\sigma_1(i)\sigma_2(j)}C_{\sigma_1(j)\sigma_2(k)}\\
&=&C_{\sigma_1(i)\sigma_1(j)}C_{\sigma_2(j)\sigma_2(k)}=e_{i, \,\sigma_2^{-1}\sigma_1(j)}e_{\sigma_1^{-1}\sigma_2(j),\,k},
\end{eqnarray*}
which implies $\sigma_2^{-1}\sigma_1(j)=\sigma_1^{-1}\sigma_2(j)$ for $1\leq j\leq n$, i.e., $(\sigma_1^{-1}\sigma_2)^2=\mathrm{id}$. While the degree of $\sigma_1^{-1}\sigma_2$ equals to that of $\sigma=\sigma_2\sigma_1^{-1}$. Hence $\sigma^2=id$.

Since $e_{ii}$ is an idempotent of $A$ for all $i$,  $C_{\sigma_1(i), \,\sigma_2(i)}$ is also an idempotent. Thus ${\rm deg}(\sigma_1(i))=-{\rm deg}(\sigma_2(i))$ for all $1\leq i\leq n$ and  $${\rm deg}(i)={\rm deg}(\sigma_1(\sigma_1^{-1}(i)))=-{\rm deg}(\sigma_2(\sigma_1^{-1}(i)))=-{\rm deg}(\sigma(i)).$$

$``\Longleftarrow"$ Firstly we prove (GC2). We need to check that the linear map $\ast$ sending $C_{ij}$ to $C_{ji}$ is an anti-morphism of $A$. Note that
$$e_{ij}^\ast=C_{\sigma_1(i),\,\sigma_2(j)}^\ast=C_{\sigma_2(j),\,\sigma_1(i)}=e_{\sigma_1^{-1}\sigma_2(j), \,\sigma_2^{-1}\sigma_1(i)},$$ which implies $$(e_{ij}e_{kl})^{\ast}=\delta_{jk}e_{il}^\ast=\delta_{jk}e_{\sigma_1^{-1}\sigma_2(l), \,\sigma_2^{-1}\sigma_1(i)}.$$ While
\begin{eqnarray*}
e_{kl}^{\ast}e_{ij}^{\ast}&=&e_{\sigma_1^{-1}\sigma_2(l), \,\sigma_2^{-1}\sigma_1(k)}e_{\sigma_1^{-1}\sigma_2(j), \,\sigma_2^{-1}\sigma_1(i)}\\&=&\delta_{\sigma_2^{-1}\sigma_1(k), \,\sigma_1^{-1}\sigma_2(j)}e_{\sigma_1^{-1}\sigma_2(l), \,\sigma_2^{-1}\sigma_1(i)}.
\end{eqnarray*}
Now $\sigma^2=\mathrm{id}=(\sigma_2^{-1}\sigma_1)$ makes $\sigma_2^{-1}\sigma_1=\sigma_1^{-1}\sigma_2$. Therefore $j=k$ if and only if $\sigma_2^{-1}\sigma_1(k)=\sigma_1^{-1}\sigma_2(j)$, that is, $\delta_{jk}=\delta_{\sigma_2^{-1}\sigma_1(k), \,\sigma_1^{-1}\sigma_2(j)}$. As a consequene,
$(e_{ij}e_{kl})^{\ast}=e_{kl}^{\ast}e_{ij}^{\ast}$. This completes the proof of (GC2).

Secondly we prove (GC3). According to the definition of $C_{ij}$, easy computations give that
$C_{ij}C_{kl}=\delta_{\sigma_2^{-1}(j),\,\sigma_1^{-1}(k)}C_{il}.$

Finally, assume that $C_{ij}C_{kl}\neq 0$. Then $\sigma_2^{-1}(j)=\sigma_1^{-1}(k)$, that is, $\sigma(j)=k$.
 Then ${\rm deg}(j)=-{\rm deg}(\sigma(j))=-{\rm deg}(k)$ implies ${\rm deg}(C_{ij}C_{kl})={\rm deg}(C_{il})$. \end{proof}

\begin{cor}\label{5.2}Keep notations as Proposition~\ref{5.1}.
If $\sigma(i)=i$ then ${\rm deg}(i)=0$.
\end{cor}

Using Proposition \ref{5.1}, we  can give a graded symmetric cellular structure of $A=M_n(K)$ as follows: Let $\Lambda=\{\diamond\}$, $M(\diamond)=\{1, 2, \cdots, n\}$, and set $C_{ij}=e_{i,n-j+1}$ for all $1\leq i,j \leq n$. Then $C_{ij}C_{kl}=\delta_{k,n-j+1}C_{il}$ shows $\{C_{ij}\mid 1\leq i, j\leq n\}$ is a cellular basis of $A$.

 Now if $n>1$ is odd then we define
\begin{equation*}
{\rm deg}(i)=
\begin{cases}
i,& i\leq \frac{n-1}{2};\\
0,& i=\frac{n+1}{2};\\
i-n-1,  & i\geq \frac{n+3}{2}.
\end{cases}
\end{equation*}
If $n$ is even  then we define
\begin{equation*}
{\rm deg}(i)=
\begin{cases}
i,& i\leq \frac{n}{2};\\
i-n-1,  & i> \frac{n}{2}.
\end{cases}
\end{equation*}
This makes $\{C_{ij}\mid 1\leq i, j\leq n\}$ a graded cellular basis of $A$. Now we define a homogeneous $K$-linear map $\tau$ from $A$ to $K$ by
\begin{equation*} \tau(C_{ij})=
\begin{cases}
1,& {\rm deg}(C_{ij})=0;\\
0,  & {\rm Otherwise.}
\end{cases}
\end{equation*}
Then $\tau(C_{ij}C_{kl})=\delta_{k,n-j+1}\delta_{i,n-l+1}=\tau(C_{kl}C_{ij})$ implies
that $\tau(ab)=\tau(ba)$ for all $a, b\in A$. Clearly $\tau$ is non-degenerate. As a Consequence,
$\tau$ is a homogeneous symmetrizing trace of degree 0.

Let us remark that the above argument can be generalized to the semisimple case, that is, we have  the following

\begin{prop}\label{5.3}
Let $A=\bigoplus M_{n_i}(K)$ be a semisimple algebra. If there exists some $n_i\neq 1$, then $A$ is a graded symmetric cellular algebra with a non-trivial grading.
\end{prop}

Now let $A$ be a semisimple algebra with a homogeneous symmetrizing trace of degree $d$ and let $\{C_{S,T}^\lam\mid\lam\in\Lambda, S,T\in M(\lam)\}$ be a homogeneous cellular basis of $A$.
Then  $k_{\lam}^{-1}C_{S,S}^\lam D_{S,S}^\lam$ is
an idempotent of $A$ (Lemma \ref{2.6}), of course its degree is 0.  While Lemma~\ref{3.2} gives that the degree of $k_{\lam}^{-1}C_{S,S}^\lam D_{S,S}^\lam$ is $-d$. Hence $d=0$ and we have proved the following lemma.

\begin{lem}\label{5.4}
The degree of all symmetrizing traces $\tau$ of semisimple algebras  is 0.
\end{lem}

\begin{lem}\label{5.5}
Let $A$ be a semisimple $K$-algebra with a homogeneous symmetrizing trace $\tau$. Then $L_{\rm gr}(A)=\mathcal {Z}_A(A_0)$.
\end{lem}

\begin{proof}
It follows from Theorem \ref{4.4} that we only need to prove $\mathcal {Z}_A(A_0)\subseteq L_{\rm gr}(A)$. Since $A$ is semisimple, we have from Proposition \ref{5.3} that $A$ is graded symmetric cellular. Let $\{C_{S,T}^{\lam}\mid\lam\in\Lambda, S, T\in M(\lam)\}$ be a homogeneous cellular basis of $A$. Note that Lemma \ref{2.8} shows $\{C_{S,S}^{\lam}D_{S,T}^{\lam}\mid S,T\in M(\lam),\lam\in\Lambda\}$ is a basis of $A$. Assume that $$a=\sum_{S,T\in M(\lam),\, \lam\in\Lambda}r_{S,T,\lam}C_{S,S}^{\lam}D_{S,T}^{\lam}\in \mathcal {Z}_A(A_0).$$
Combining Lemmas \ref{5.4} and \ref{3.2}, $\deg (C_{X,X}^{\epsilon}D_{X,X}^{\epsilon})=0$  and $C_{X,X}^{\epsilon}D_{X,X}^{\epsilon}\in A_0$ for  $\epsilon\in \Lambda$, $X\in M(\epsilon)$.
Therefore \begin{eqnarray*}
\sum\limits_{P\in M(\epsilon)}r_{P,X,\epsilon}k_{\epsilon}C_{P,P}^{\epsilon}D_{P,X}^{\epsilon}
   &=&\sum_{\substack{\lam\in\Lambda,\\S,T\in M(\lam)}} r_{S,T,\lam}C_{S,S}^{\lam}D_{S,T}^{\lam}C_{X,X}^{\epsilon}D_{X,X}^{\epsilon}\\
  &=& C_{X,X}^{\epsilon}D_{X,X}^{\epsilon}\sum_{\substack{\lam\in\Lambda,\\S,T\in M(\lam)}}r_{S,T,\lam}C_{S,S}^{\lam}D_{S,T}^{\lam} \\
               &=& \sum\limits_{Q\in M(\epsilon)}r_{X,Q,\epsilon}k_{\epsilon}C_{X,X}^{\epsilon}D_{X,Q}^{\epsilon},
                     \end{eqnarray*}
where the first equality and the last one follow by Lemmas~\ref{2.4} and \ref{2.6}.

  Since $A$ is semisimple, $k_{\epsilon}\neq 0$ according to Lemma \ref{2.8}. Thus $r_{P, X, \epsilon}=0$ if $P\neq X$ for $\epsilon\in \Lambda$, $P, X\in M(\lam)$, that is,
$a=\sum\limits_{\lam\in\Lambda, \,S\in M(\lam)}r_{S,\lam}C_{S,S}^{\lam}D_{S,S}^{\lam}.$

Now assume that $P, Q\in M(\epsilon)$ and $\deg(P)=\deg(Q)$. Then Lemma \ref{3.2} and Lemma \ref{5.4} imply $\deg(C_{P,P}^{\epsilon}D_{P,Q}^{\epsilon})=0$.  Thus $aC_{P,P}^{\epsilon}D_{P,Q}^{\epsilon}=C_{P,P}^{\epsilon}D_{P,Q}^{\epsilon}a$. By employing
the same argument as above, one obtain $r_{P, \epsilon}=r_{Q, \epsilon}$. So $a\in L_{\mathrm{gr}}(A)$, that is, $\mathcal {Z}_A(A_0)\subseteq L_{\rm gr}(A)$.
\end{proof}

\begin{lem}
Let $A$ be a finite dimensional graded symmetric cellular algebra. If $L_{\rm gr}(A)=\mathcal {Z}_A(A_0)$ then $A$ is semisimple.
\end{lem}

\begin{proof}
 Since $L_{\rm gr}(A)=\mathcal {Z}_A(A_0)$, we assume that $1=\sum\limits_{\varepsilon\in\Lambda, \,c \,\in\,\Z} r_{\varepsilon, \,c}e_{\varepsilon, \,c}$. For $\lam\in\Lambda$, $S\in M(\lam)$, Lemma~\ref{4.2'} implies $e_{\lam, \,c_0}=r_{\lam, \,c_0}k_{\lam}e_{\lam, \,c_0}$, where $c_0=\deg(S)$. Clearly, $e_{\lam, \,c_0}\neq 0$ and consequently $k_{\lam}\neq 0$. Thus $A$ is semisimple owing  to Lemma \ref{2.8}.
\end{proof}

Combining Lemmas 5.5 and 5.6, we obtain the main result of this section.

\begin{thm}\label{Them:semisimple}
Let $A$ be a finite dimensional graded symmetric cellular algebra. Then $A$ is semisimple if and only if $L_{\rm gr}(A)=\mathcal {Z}_A(A_0)$.
\end{thm}

\subsection*{Acknowledgements}
%Both authors would like to thank the reviewers for the valuable suggestions
%which helped to clarify the whole paper.
The authors thank the Chern Institute of Mathematics in Nankai University for the hospitality during their visits.

\end{document}